\documentclass[12pt]{article}

\usepackage{datetime}

\usepackage[all]{xy}

\usepackage{amssymb, amsmath, amsthm,eucal}

\usepackage{mathptmx}
\usepackage[scaled=.90]{helvet}
\usepackage{courier}

\usepackage[pdftex]{hyperref}

\hypersetup{
  colorlinks = true,
  linkcolor = black,
  urlcolor = blue,
  citecolor = black,
}

\renewcommand{\H}{\mathrm{H}}
\renewcommand{\mod}{\;\mathrm{mod}\;}
\renewcommand{\dim}{\operatorname{dim}}

\newcommand{\Zset}{\mathbb{Z}}

\newcommand{\Rset}{\mathbb{R}}
\newcommand{\Cset}{\mathbb{C}}

\newtheorem{thm}{Theorem}

\newtheorem{prop}{Proposition}

\usepackage{graphicx}

\DeclareRobustCommand{\coprod}{\mathop{\text{\fakecoprod}}}
\newcommand{\fakecoprod}{%
  \sbox0{$\prod$}%
  \smash{\raisebox{\dimexpr.9625\depth-\dp0}{\scalebox{1}[-1]{$\prod$}}}%
  \vphantom{$\prod$}%
}

\title{\bf A stable approach to the \hbox{equivariant Hopf theorem}}

\author{Markus Szymik}

\date{March 2007}

\begin{document}

\maketitle

\begin{abstract}
\noindent
  Let~$G$ be a finite group. For semi-free~$G$-manifolds
  which are oriented in the sense of Waner~\cite{Waner}, the
  homotopy classes of~$G$-equivariant maps into a~$G$-sphere
  are described in terms of their degrees, and the degrees
  occurring are characterised in terms of congruences. This
  is first shown to be a stable problem, and then solved
  using methods of equivariant stable homotopy theory with
  respect to a semi-free~$G$-universe.
%  \newline
%  \newline
%  \noindent{\sc Keywords.} 
%  $G$-manifold,
%  degree,
%  equivariant stable homotopy.
%  \newline
%  \newline
%  \noindent{\sc MSC.} 
%  57R91, % equivariant algebraic topology of manifolds
%  55P91, % equivariant homotopy theory
%  55P92.  % relations between equivariant and nonequivariant homotopy theory
\end{abstract}

\thispagestyle{empty}

%%%%%%%%%%%%%%%%%%%%%%%%%%%%%%%%%%%%%%%%%%%%%%%%%%%%%%%%%%%%%%%%%%%%%%%%%%%%%

\section*{Introduction}

All manifolds considered will be smooth, closed, and
connected. If~$M$ is an oriented~$n$-manifold, any
continuous map from~$M$ to an~\hbox{$n$-dimensional} sphere
has a degree, which is an integer.  The Hopf
theorem~\cite{Hopf} says that the degree
cha\-rac\-te\-ri\-ses the homotopy class of such a map, and
that every integer occurs as the degree of such a map. Note
that the degree map from homotopy classes to the integers
factors through the group~$\{M,S^n\}$ of {\it stable}
homotopy classes of maps from~$M$ to~$S^n$. The Hopf theorem
follows from the facts that both the stabilisation and the
{\it stable} degree map are isomorphisms. It is this
approach to that result which will be generalised here to
the equivariant category: it will first be shown that the
problem, which at first sight seems to be unstable, is in
fact stable, and then it will be solved using stable
techniques.

The Hopf theorem has been generalised to equivariant
contexts by a number of people, motivated by the problem to
describe maps between representation spheres or homotopy
representations. See Section
8.4.~in~\cite{tomDieck:LNM},~\cite{Hauschild:Zerspaltung},
Section II.4 in~\cite{tomDieck:Book},
or~\cite{BalanovKushkuley}, \cite{Balanov},~\cite{Ferrario},
and the references therein for recent contributions. All of
these work unstably with a fairly elaborate obstruction
machinery. Furthermore, they require the top-dimensional
cohomology of the fixed point spaces to be cyclic,
neglecting situations in which the fixed point space is not
connected. The point of view taken here avoids elaborate
notions of degree (and corresponding Lefschetz and Euler
numbers) as in~\cite{Dold},
\cite{Ulrich},~\cite{Lueck},~\cite{Ize}, and the references
therein, to name just a few. For the present purpose, the
degree of a map is its stable homotopy class, and
calculations of stable homotopy groups -- which are much
more accessible than their unstable relatives -- process
this into numerical information.

Throughout the paper, let~$G$ be a finite group. There will
be occasions to assume that its order is a prime, but only
for illustrative purposes. Let~$W$ be a
real~$G$-representation. Recall that a~$G$-manifold~$M$ is
called~{\it~$W$-dimensional}, if for every point~$x$ in~$M$
the tangential representation~$T_xM$ of the stabiliser~$G_x$
is isomorphic to the restriction of~$W$ from~$G$
to~$G_x$. In particular, the components of the fixed point
space
\begin{displaymath}
  M^G=\coprod_\alpha M^G_\alpha
\end{displaymath}
are manifolds of equal dimension. In addition, a
$W$-dimensional manifold~$M$ is called {\it oriented} if~$W$
is oriented and there is a compatible choice of
isomorphisms~$T_xM\cong W$ which are orientation
preserving. In particular, the components of the fixed point
space inherit an orientation. (See~\cite{Waner} and the
references therein for details.)

{\bf Example~1.} If $M\rightarrow N$ is a cyclic branched
covering of complex manifolds, the representation $W$ is a
direct sum of trivial summands and a complex line on which
the cyclic group acts in the natural way. \hfill\qedsymbol

The following example has been considered and applied
in~\cite{Szymik:Galois}. Note that contrary to the previous
example, the fixed point space is not connected here.

{\bf Example~2.} For a prime order group~$G$, and a
complex~$G$-repre\-sen\-tation~$V$, let~$M=\Cset P(V)$ be
the associated projective space. The fixed point space is
the disjoint union of the projective spaces of the
isotypical summands, hence is equidimensional if~$V$ is a
multiple of the regular~$G$-representation. In that case,
all tangential representations are isomorphic, as
complex~$G$-representations, to the complement of a trivial
line in~$V$. This we may take as our~$W$.  \hfill\qedsymbol

In this writing, we will be concerned with
oriented~$W$-dimensional semi-free~$G$-manifolds~$M$, and
prove an equivariant Hopf theorem which gives a description
of the set of~$G$-homotopy classes of~$G$-maps from~$M$
to~$S^W$ in the generic case when the dimension and
codimension of the fixed point space are at least~$2$. (See
Section~\ref{sec:comparison} for free actions; trivial
actions can be dealt with as in Example~1 below.) By
orientability, the codimension of the fixed point space is
always at least two if the action is non-trivial.

While the Hopf problem looks like an unstable problem, its
turns out to be stable. Equivariant stable homotopy theory
is more complicated than ordinary stable homotopy theory, as
there is a stable homotopy category for each~$G$-{\it
universe}, which is an
infinite-dimensional~$G$-representation which contains the
trivial representation, and contains each of its
subrepresentation with infinite multiplicity.~(Standard
references for equivariant stable homotopy theory
are~\cite{Adams},~\cite{LMS},~\cite{Mayetal},
and~\cite{GreenleesMay}.)  The full-flavoured equivariant
stable homotopy category corresponds to a {\it
complete}~$G$-universe, which contains every
irreducible~$G$-representation. A {\it semi-free} universe
is obtained from a complete~$G$-universe by restriction to
the semi-free subrepresentations. The stable homotopy
category depends only on the~$G$-isomorphism class of the
universe, and we will
write~$\{X,Y\}^G_{\mathrm{\mathrm{sf}}}$ for the morphism
groups from~$X$ to~$Y$ in an equivariant stable homotopy
category corresponding to a semi-free~$G$-universe. As the
stable categories requires based objects, we will let
$M_+$ denote the $G$-space which is $M$ with a disjoint
$G$-fixed basepoint added. By adjunction, there is a
canonical bijection between the set $[M_+,S^W]^G$ of based
$G$-homotopy classes and the set of (unbased) $G$-homotopy
classes of maps from $M$ to $S^W$. The proof of the
following result is in Section~\ref{sec:stability}. It
justifies working stably afterwards.

\parbox{\linewidth}{\begin{thm}\label{thm:stability} Let~$W$
  be a~$G$-representation, and let~$M$ be a~$W$-dimensional
  semi-free~$G$-manifold such that the dimension and
  codimension of the fixed point space are at
  least~$2$. Then the stabilisation map
  \begin{displaymath}
    [M_+,S^W]^G\longrightarrow\{M_+,S^W\}^G_{\mathrm{\mathrm{sf}}}
  \end{displaymath}
  is bijective.
\end{thm}}
 
In order to study the stable
groups~\hbox{$\{M_+,S^W\}^G_{\mathrm{\mathrm{sf}}}$}, the
standard approach is to associate to a stable~$G$-map~$f$
the set of all the degrees of the fixed point maps~$f^H$ for
the various subgroups~$H$ of~$G$. In the semi-free case, the
only subgroups relevant are~$1$ and~$G$, so that the fixed
point maps can be organised into the {\it ghost map}
\begin{equation}\label{ghost map}
  \{M_+,S^W\}^G_{\mathrm{\mathrm{sf}}}
  \longrightarrow 
  \{M_+,S^W\}
  \oplus
  \{M^G_+,S^{W^G}\}
\end{equation}
which sends a stable~$G$-map~$f$ to the pair~$(f,f^G)$. A
priori, the first factor of the ghost map lies in
the~$G$-invariants of~$\{M_+,S^W\}$, but the present
assumptions on the~$G$-actions on~$M$ and~$W$ ensure
that~$G$ acts trivially on this group. Using the ordinary
Hopf theorem, the image of an~$f$ under~(\ref{ghost map})
can be interpreted in terms of degrees: an integer~$x$ and a
sequence~$(y_\alpha)$ of integers, indexed by the components
of~$M^G$. (See Section~\ref{sec:examples} for details.) One
may show -- using standard splitting and localisation
techniques -- that the ghost map~(\ref{ghost map}) is an
isomorphism away from the order of~$G$, but this will also
be shown directly in the course of this investigation. It
means that the kernel and cokernel are finite abelian
groups. The following summarises
Propositions~\ref{prop:structure of the source}
and~\ref{prop:ghost map is injective} of the main part.

\begin{thm}
  \label{thm:structure}
  Let~$M$ be an oriented~$W$-dimensional~$G$-manifold. The
  source and the target of the ghost
  map~{\upshape(\ref{ghost map})} are free abelian. Their
  rank is one larger than the number of components
  of~$M^G$. The map is injective and the cokernel is cyclic
  of order~$|G|$.
\end{thm}
  
In other words, stable~$G$-maps are determined by their
degrees, and the numbers that occur satisfy a relation. The
following result, which is proved in
Section~\ref{sec:examples}, says which.

\parbox{\linewidth}{\begin{thm}
  \label{thm:congruences} 
  The image of the ghost map~{\upshape(\ref{ghost map})} is
  given by the subgroup of integers~$x$ and
  sequences~$(y_\alpha)$ for which the congruence
  \begin{displaymath}
    x\equiv\sum_{\alpha}y_{\alpha}\mod |G|
  \end{displaymath}
  is satisfied.
\end{thm}}

The chosen approach to the two theorems separates the
homotopy theory from the geo\-metry. Let me illustrate this
by explaining how the well-known description of the Burnside
rings of prime order groups can be deduced using these
methods.

{\bf Example~3.} Let~$G$ be a prime order group, and let~$M$
be a point. Then~$M_+$ and~$S^W$ are both 0-dimensional
spheres.  While the~$G$-actions on~$M$ is trivial, of
course, it will no longer be after stabilisation with a
suitable oriented~$G$-representation. Let us assume that
this has been done without changing the notation. The
group~$\{S^0,S^0\}^G$ maps via the forgetful map to the
group~$\{S^0,S^0\}=\Zset$.  On the other hand, it maps to
the group~$\{(S^0)^G,(S^0)^G\}=\Zset$ via the fixed point
map.  The product
\begin{displaymath}
  \{S^0,S^0\}^G\longrightarrow\Zset\oplus\Zset
\end{displaymath}
of the two maps is the ghost map. It is injective, and the
image has index~$|G|$.  The image is not yet determined by
this. There are still~$|G|+1$ choices. In order to give a
description of the image, one needs to put in some new
information. One can use the fact that the identity of~$S^0$
is~$G$-equivariant. It is mapped to the pair~$(1,1)$. This
determines the image, which consists of those pairs~$(x,y)$
in~$\Zset\oplus\Zset$ which satisfy the relation~$x\equiv
y\mod |G|$.  \hfill\qedsymbol

The homotopy theory for Theorem~\ref{thm:structure} is done
in Sections \ref{sec:3} to \ref{sec:7}. The method of proof
is to deal with the~$G$-trivial space~$M^G$ and with
the~$G$-space~$M/M^G$, which is free away from the base
point, separately. The latter uses a result from
Section~\ref{sec:comparison} which sometimes allows for a
comparison of the group of~$G$-maps out of a free~$G$-space
with the group of ordinary maps out of the quotient in the
case when~$G$ does {\it not} act trivially on the target. In
the final Section~\ref{sec:examples}, the geometry of the
situation is used to construct some maps whose images under
the ghost map are easy to determine, leading to
Theorem~\ref{thm:congruences}.
 
\subsection*{Acknowledgments} 

I would like to thank a referee
of an earlier version for her or his stimulating report. Part of this work has been supported by the SFB 343 at the University of Bielefeld.

%%%%%%%%%%%%%%%%%%%%%%%%%%%%%%%%%%%%%%%%%%%%%%%%%%%%%%%%%%%%%%%%%%%%%%%%%%%%%

\section{Stability}\label{sec:stability}

In this section, we will prove the stability of the Hopf problem, and
the proof given will also show that the semi-stable universe is the
appropriate context in which to work. 

Let us recall the classical situation. If~$M$ is an
oriented~$n$-manifold, by Freudenthal's suspension theorem, the
stabilisation map
\begin{displaymath}
  [M_+,S^n]\longrightarrow\{M_+,S^n\}
\end{displaymath}
is bijective if~$n\geqslant 2$. In the equivariant
situation, we may use the equivariant extension of
Freudenthal's theorem due to H. Hauschild, see
\cite{Hauschild:Suspension} (or~\cite{Adams} for an
exposition in English). This implies that the suspension map
\begin{displaymath}
  [X,Y]^G\longrightarrow[\Sigma^VX,\Sigma^VY]^G
\end{displaymath}
is bijective for finite based $G$-CW-complexes $X$ and $Y$,
if two conditions are satisfied: the connectivity of $Y^H$
has to be at least $\dim(X^H)/2$ for all subgroups~$H$
of~$G$ such that~\hbox{$V^H\not=0$}, and the connectivity of
$Y^K$ has to be at least $\dim(X^H)+1$ for all
subgroups~$K<H$ of~$G$ such that~$V^H\not=V^K$. This will
now be used to prove Theorem~\ref{thm:stability}.

\begin{proof}
  The first condition corresponds to Freudenthal's condition
  in the non-equivariant case. It is therefore satisfied for
  all~$H$ -- and for all~$V$ -- in our situation~$X=M_+$
  and~$Y=S^W$ as long as the dimension of~$M^G$ is at
  least~$2$.

  The other condition refers to a genuinely equivariant
  phenomenon.  Suppose first that~$K=1$ and~$H$ is a
  non-trivial subgroup of~$G$. Then~$M^H=M^G$ since~$M$ is
  semi-free, and the condition is satisfied for these~$K$
  and~$H$ -- and for all~$V$ -- if and only if the
  codimension of~$M^G$ in~$M$ is at least~$2$.

  If~$K$ and~$H$ are non-trivial subgroups of~$G$ such
  that~$V^H\not=V^K$, the
  condition~$\dim(M^G)\leqslant\dim(W^G)-2$ needs to be
  satisfied, which does not seem to be the case. However, we
  always have~$V^H=V^G=V^K$ if~$V$ is semi-free, so this
  does not occur, and the proof is finished.
\end{proof}

There is a change-of-universe natural transformation
from~\hbox{$\{X,Y\}^G_{\mathrm{\mathrm{sf}}}$} to the corresponding
group with respect to a complete~$G$-universe,
see~\cite{Lewis:Change}. However, the preceding proof shows that one
should not expect it to be bijective. It is clearly bijective if the
order of~$G$ is a prime, since semi-free universes are complete for
such~$G$. In a similar vein, the proof above can easily be adapted to
show that the stabilisation map from $[M_+,S^W]^G$ to the
corresponding group of stable $G$-maps with respect to a complete
$G$-universe is always bijective if $M$ satisfies a suitable gap
hypothesis: The fixed point space $M^G$ has to be at least
$2$-dimensional, and the codimension of $M^H$ in $M^K$ has to be at
least $2$ for all subgroups $K<H$ of $G$.

%%%%%%%%%%%%%%%%%%%%%%%%%%%%%%%%%%%%%%%%%%%%%%%%%%%%%%%%%%%%%%%%%%%%%%%%%%%%%

\section{A comparison result}\label{sec:comparison}
 
Let~$G$ be a finite group. A based~$G$-space is called {\it
free} if the group~$G$ acts freely on the complement of the
base point (which is fixed by $G$). Let~$F$ be a finite free
based~$G$-CW-complex. Since~$G$ is finite, this also yields
an ordinary CW-structure on~$F$. And it gives an ordinary
CW-structure on the quotient~$Q=F/G$.

If~$Y$ is a~$G$-space, the group~$\{F,Y\}^G_{\mathrm{sf}}$
is not obviously related to~$\{Q,Y\}$ since the~$G$-action
on~$Y$ need not be trivial. If~$Y$ were a trivial~$G$-space,
there would be a tautological map from~$\{Q,Y\}$
to~$\{Q,Y\}^G_{\mathrm{sf}}$. One could then use a semi-free
version of (5.3) from~\cite{Adams}, which says that in this
case the composition~\hbox{$\{Q,Y\}\rightarrow
\{Q,Y\}^G_{\mathrm{sf}} \rightarrow\{F,Y\}^G_{\mathrm{sf}}$}
with the map induced by the projection from~$F$ to~$Q$ would
be an isomorphism. However, since the action on~$Y$ is
non-trivial, the arrow on the left is not
defined. Nevertheless, there sometimes is a way to
compare~$\{F,Y\}^G_{\mathrm{sf}}$ with~$\{Q,Y\}$. This will
be explained now.

By mapping the CW-filtrations into~$Y$, one gets three
spectral sequences, which converge
to~$\{F,Y\}^G_{\mathrm{sf}}$,~$\{F,Y\}$ and~$\{Q,Y\}$,
respectively.

For an integer~$s$, let~$I(s)$ denote the (finite) set
of~$s$-dimensional~$G$-cells in~$F$. The filtration gives
rise to a spectral sequence
\begin{equation}\label{1st_ss} 
  E_1^{s,t}=\bigl\{\bigvee_{I(s)}G_+,\Sigma^tY\bigr\}^G_{\mathrm{sf}}
  \Longrightarrow \{F,\Sigma^{s+t}Y\}^G_{\mathrm{sf}}.
\end{equation} 
In order to compute the groups on the~$E_2$-page, the
differentials on the~$E_1$-page need to be
discussed. The~$t$-th row is obtained by applying the
functor~$\{?,\Sigma^tY\}^G_{\mathrm{sf}}$ to
the~$G$-cellular complex
\begin{displaymath} 
  *\longleftarrow \bigvee_{I(0)}G_+\longleftarrow 
  \bigvee_{I(1)}G_+\longleftarrow\dots\longleftarrow 
  \bigvee_{I(n)}G_+\longleftarrow * 
\end{displaymath} 
of~$F$. 
Note that 
\begin{equation}\label{matrix_ref} 
  \bigl\{\bigvee_{I(s)}G_+,\Sigma^tY\bigr\}^G_{\mathrm{sf}} \cong 
  \bigoplus_{I(s)}\bigl\{G_+,\Sigma^tY\bigr\}^G_{\mathrm{sf}}, 
\end{equation} 
and that the groups~$\{G_+,\Sigma^tY\}^G_{\mathrm{sf}}$ are
right modules over~$\{G_+,G_+\}^G_{\mathrm{sf}}$ via
composition. Because of~(\ref{matrix_ref}), the
differentials can be identified with matrices, and the
entries are given by right multiplication with elements
of~$\{G_+,G_+\}^G_{\mathrm{sf}}$.  Therefore, let me pause
to discuss this action in more detail.
 
Sending an element~$g$ of~$G$ to the~$G$-map~$G_+\rightarrow
G_+$ which sends an element~$x$ to~$xg$ induces an
isomorphism of the group ring~$\Zset G$
with~$\{G_+,G_+\}^G_{\mathrm{sf}}$ which reverses the order
of the
multiplication. Therefore,~$\{G_+,\Sigma^tY\}^G_{\mathrm{sf}}$
is a left~$\Zset G$-module. As such, it is isomorphic with
the left~$\Zset G$-module~$\{S^0,\Sigma^tY\}$,
the~$G$-action being induced by the action on~$Y$. (The
adjunction isomorphism (5.1) in~\cite{Adams} is~$G$-linear.)
This finishes the digression on the~$\Zset G$-module
structures.
 
One can now try to compare the spectral
sequence~(\ref{1st_ss}) to the one
\begin{equation}\label{2nd_ss} 
  E_1^{s,t}=\bigl\{\bigvee_{I(s)}S^0,\Sigma^tY\bigr\} 
  \Longrightarrow
  \{Q,Y\}^{s+t}
\end{equation} 
obtained by the induced CW-filtration of~$Q$. (This is of course the
Atiyah-Hirze\-bruch spectral sequence for~$Y$-cohomology.) As mentioned above, there is no
reasonable map between the targets in sight, and~I cannot offer a map
of spectral sequences. However, note that the groups on
the~$E_1$-pages are always
isomorphic:~\hbox{$\{G_+,\Sigma^tY\}^G_{\mathrm{sf}}
\cong\{S^0,\Sigma^tY\}$}, again by the adjunction isomorphism~(5.1)
in~\cite{Adams}. As for the differentials, the following is true.

\begin{prop}\label{prop:differentials}
  If~$\{S^0,\Sigma^tY\}$ is a trivial~$\Zset G$-module,
  the~$t$-rows on the~$E_1$-terms of the spectral sequences
  {\upshape (\ref{1st_ss})} and {\upshape (\ref{2nd_ss})}
  are isomorphic as complexes. In particular, the groups on
  the~$t$-rows of the~$E_2$-pages are isomorphic.
\end{prop}
 
\begin{proof}
  The differentials on the~$E_1$-page of the spectral
  sequence~(\ref{2nd_ss}) are obtained by applying the
  functors~$\{?,\Sigma^tY\}$ to the cellular complex
  \begin{displaymath} 
    *\longleftarrow \bigvee_{I(0)}S^0\longleftarrow 
    \bigvee_{I(1)}S^0\longleftarrow \dots\longleftarrow 
    \bigvee_{I(n)}S^0\longleftarrow * 
  \end{displaymath} 
  of~$Q$. As for~(\ref{1st_ss}), they can be thought of as
  matrices.  This time, the entries are obtained from the
  entries of those in~(\ref{1st_ss}) by passage to
  quotients, i.e. by applying the
  map~$\epsilon:\{G_+,G_+\}^G_{\mathrm{sf}}
  \rightarrow\{S^0,S^0\},\;g\mapsto1$. This means that the
  diagram
  \begin{center}
    \mbox{ 
      \xymatrix{
        \{\bigvee_{I(s)}G_+,\Sigma^tY\}^G_{\mathrm{sf}}\ar[r] & 
        \{\bigvee_{I(s+1)}G_+,\Sigma^tY\}^G_{\mathrm{sf}}\\ 
        \{\bigvee_{I(s)}S^0,\Sigma^tY\}\ar[r]\ar[u]^{\cong} & 
        \{\bigvee_{I(s+1)}S^0,\Sigma^tY\}\ar[u]_{\cong} 
      }
    } 
  \end{center}  
  (obtained from the isomorphisms above and the
  differentials) is only commutative if the elements
  in~$\Zset G$ which appear in the matrix of the top arrow
  act via~$\epsilon$. While at first sight it seems that one
  needs to know the details of the~$G$-CW-structure to
  proceed, this is not the case: if~$\{S^0,\Sigma^tY\}$ is a
  trivial~$\Zset G$-module, the condition is fulfilled for
  all elements of~$\Zset G$.
\end{proof}

In nice situations, this result implies
that~$\{F,Y\}^G_{\mathrm{sf}}$ and~$\{Q,Y\}$ are in fact
isomorphic. One of these situations will be encountered in
the following section.
 
%%%%%%%%%%%%%%%%%%%%%%%%%%%%%%%%%%%%%%%%%%%%%%%%%%%%%%%%%%%%%%%%%%%%%%%%%%%%%

\section{Free points of spheres}\label{sec:3}
 
Let~$W$ be a orientable real~$G$-representation which in
non-trivial and semi-free. The results of the previous
section will now be applied to the
free~$G$-space~\hbox{$F=S^W/S^{W^G}$} and~$Y=S^W$. Let~$n$
and~$n^G$ be the real dimensions of~$S^W$ and~$S^{W^G}$,
respectively. Note that the number~$n-n^G$ is positive and
even.
 
\begin{prop}
  The three groups
  \begin{displaymath}
  \{S^W/S^{W^G},S^W\},\quad\{(S^W/S^{W^G})/G,S^W\},\quad\text{ and
  }\quad\{S^W/S^{W^G},S^W\}^G_{\mathrm{sf}}
  \end{displaymath}
  are all free abelian on one generator.
\end{prop}

\begin{proof}
  As for the group~$\{S^W/S^{W^G},S^W\}$: Since
  the~$G$-representation~$W$ is non-trivial, one knows that
  the quotient~$S^W/S^{W^G}$ is non-equivariantly equivalent
  to a wedge~\hbox{$S^n\vee S^{n^G+1}$}.  The map from the
  group~$\{S^W/S^{W^G},S^W\}$ to~$\{S^W,S^W\}=\Zset$ induced
  by the collapse map is an isomorphism. Note
  that~$\{S^W/S^{W^G},S^W\}$ is isomorphic
  to~$\H^n(S^W/S^{W^G};\Zset)$ via the Hurewicz map.
  
  As for~$\{(S^W/S^{W^G})/G,S^W\}$: Also via the Hurewicz
  map, this group is isomorphic to the
  group~$\H^n((S^W/S^{W^G})/G;\Zset)$.  Therefore, one can
  use the spectral sequence
  \begin{equation}\label{hypercohomology} 
    E_2^{s,t}=\H^s(G;\H^t(S^W/S^{W^G};\Zset))
    \Longrightarrow 
    \H^{s+t}((S^W/S^{W^G})/G;\Zset) 
  \end{equation} 
  for that. From the stable homotopy type of~$S^W/S^{W^G}$
  it follows that the only non-trivial groups on
  the~$E_2$-page are in two rows:~$t=n^G+1$ and~$t=n$. Each
  of these contains the cohomology of the group~$G$ with
  coefficients in the trivial~$G$-module~$\Zset$.
  
  Thus there is only one page on which non-trivial
  differentials may occur. Since the dimension
  of~$(S^W/S^{W^G})/G$ is at most~$n=\dim_\Rset(W)$, all
  differentials between non-trivial groups must be
  non-trivial. This determines the~$E_{\infty}$-page. There
  are no extension problems. It follows
  that~\hbox{$\{(S^W/S^{W^G})/G,S^W\}\cong\Zset$}. But,
  notice the edge homomorphism of the spectral sequence,
  which is induced by the quotient map from~$S^W/S^{W^G}$
  to~$(S^W/S^{W^G})/G$. It can immediately be read off that
  the generator of~$\H^n((S^W/S^{W^G})/G;\Zset)$ is mapped
  to~$|G|$ times the generator
  of~$\H^n(S^W/S^{W^G};\Zset)$. Of course, this just reminds
  us that the quotient map has degree~$|G|$.
  
  As for~$\{S^W/S^{W^G},S^W\}^G_{\mathrm{sf}}$,
  Proposition~\ref{prop:differentials} from the previous
  section can be used. The assumption on the action is
  satisfied for~$Y=S^W$ and all~$t$ since the~$G$-action
  on~$W$ preserves the orientation.  We can thus deduce some
  of the groups on the~$E_2$-page of the spectral
  sequence~(\ref{1st_ss}) from the previously calculated
  groups on the~$E_2$-page of the spectral
  sequence~(\ref{2nd_ss}): the
  groups~$\H^s((S^W/S^{W^G})/G\,;\pi^t(S^W))$ vanish
  if~\hbox{$s>n$} or~\hbox{$t>0$}, and hence the only
  pair~$(s,t)$ with~\hbox{$s+t=n$} and~$E_2^{s,t}\neq0$ is
  the pair~\hbox{$(s,t)=(n,0)$}.  The corresponding
  group~$\H^n((S^W/S^{W^G})/G\,;\Zset)$ has been shown to be
  isomorphic to~$\Zset$. There are no non-trivial
  differentials into and out of it, so it survives
  to~$E_{\infty}$.  There are no extension problems.
\end{proof}

Now that the structure of those three groups is known, it is
desirable to know the maps between them. The proof of the
preceding proposition shows that the map
\begin{displaymath}
  \{(S^W/S^{W^G})/G,S^W\}\rightarrow\{S^W/S^{W^G},S^W\}
\end{displaymath}
is injective with cyclic cokernel of order~$|G|$. The same
holds for the forgetful map:
 
\parbox{\linewidth}{\begin{prop}\label{prop:forgetfulmap}
  The forgetful map
  \begin{displaymath} 
    \{S^W/S^{W^G},S^W\}^G_{\mathrm{sf}}\longrightarrow\{S^W/S^{W^G},S^W\} 
  \end{displaymath} 
  is injective and has a cyclic cokernel of order~$|G|$.
\end{prop}}

\begin{proof} 
  Let us contemplate the following diagram. 
  \begin{center}
    \mbox{ 
      \xymatrix{
        \{S^{W^G},S^W\}^G_{\mathrm{sf}} & 
	\{S^W,S^W\}^G_{\mathrm{sf}}\ar[d]\ar[l] & 
        \{S^W/S^{W^G},S^W\}^G_{\mathrm{sf}}\ar[d]\ar[l] \\ & 
        \{S^W,S^W\} & 
        \{S^W/S^{W^G},S^W\}\ar[l]_-{\cong} 
      }
    } 
  \end{center}
  The horizontal maps are induced by the obvious cofibre
  sequence.  Therefore, the top row is exact. The vertical
  maps are the forgetful maps and make the diagram
  commute. Now if one picks a map
  in~$\{S^W/S^{W^G},S^W\}^G_{\mathrm{sf}}$, its image
  in~$\{S^W,S^W\}^G_{\mathrm{sf}}$ has degree zero when
  restricted to the fixed point spheres. It follows that the
  degree of the image itself is a multiple of~$|G|$ by
  following proposition.
\end{proof}  

%It follows -- for example using equivariant complex~\hbox{$K$-theory} as in~II.5
%of~\cite{tomDieck} or using Borel cohomology -- that the degree of the image
%itself is a multiple of~$p$.

\begin{prop}
  Let~$f\colon S^W\rightarrow S^W$ be~$G$-equivariant for
  some finite group~$G$. If the degree of the restriction
  of~$f$ to the fixed sphere is zero, then the degree of~$f$
  itself is divisible by the order of~$G$.
\end{prop}

\begin{proof} 
  We will use Borel
  cohomology~$b^*(X)=\H^*(EG_+\wedge_GX;\Zset/|G|)$ with
  co-efficients
  in~$\Zset/|G|$. The~$b^*$-modules~$b^*(S^{W^G})$
  and~$b^*(S^W)$ are free of rank~$1$ by the suspension
  theorem and the Thom isomorphism, respectively. The
  inclusion~$S^{W^G}\subseteq S^W$ induces an
  inclusion~$b^*(S^W)\subseteq b^*(S^{W^G})$ since the
  quotient~$S^W/S^{W^G}$ is free and therefore
  has~$b^*$-torsion Borel cohomology. By hypothesis, the
  map~$f$ induces to zero map on~$b^*(S^{W^G})$, and so it
  has to on~$b^*(S^W)$.
\end{proof}  

%%%%%%%%%%%%%%%%%%%%%%%%%%%%%%%%%%%%%%%%%%%%%%%%%%%%%%%%%%%%%%%%%%%%%%%%%%%%%

\section{Free points in general} 
 
Let~$W$ be an orientable real~$G$-representation which is
non-trivial and semi-free, as before. The results of the
previous section will now be generalised from~$S^W$ to
oriented~$W$-dimensional~$G$-manifolds~$M$.

\parbox{\linewidth}{\begin{prop}\label{prop:freepoints}
  Source and target of the forgetful map
  \begin{displaymath}
    \{M/M^G,S^W\}^G_{\mathrm{sf}}\longrightarrow\{M/M^G,S^W\}
  \end{displaymath}
  are free abelian on one generator. The map is injec\-tive
  with a cyclic cokernel of order~$|G|$.
\end{prop}}
 
\begin{proof}
  Choose an orientation preserving~$G$-embedding of~$W$ onto
  a neighbourhood of a fixed point. Let~$A$ be the
  complement of the image. The collapse map from~$M$
  to~\hbox{$M/A\cong S^W$} and the forgetful maps induce a
  commutative diagram
  \begin{center}
    \mbox{ 
      \xymatrix{
        \{M/M^G,S^W\}^G_{\mathrm{sf}}\ar[d] & 
        \{S^W/S^{W^G},S^W\}^G_{\mathrm{sf}}\ar[d]\ar[l] \\ 
        \{M/M^G,S^W\} & 
        \{S^W/S^{W^G},S^W\}.\ar[l]
      }
    } 
  \end{center}
  By Proposition~\ref{prop:forgetfulmap} it is sufficient to
  show that the two horizontal maps are isomorphisms. The
  bottom row is isomorphic to
  \begin{displaymath}
        \{M_+,S^W\} \longleftarrow \{S^W,S^W\},
  \end{displaymath}
  which is an isomorphism by the ordinary Hopf theorem. As
  for the top row, it is easy to see that it is surjective:
  the fibre of the inclusion of~$M/M^G$ in to~$S^W/S^{W^G}$
  is~$A/A^G$ which is~$G$-free and cohomologically at
  most~$(n-1)$-dimensional.  It follows
  that~\hbox{$\{A/A^G,S^W\}^G_{\mathrm{sf}}$} is trivial. It
  does not follow, however,
  that~\hbox{$\{\Sigma(A/A^G),S^W\}^G_{\mathrm{sf}}$} is
  trivial, too. But injectivity of the top arrow follows
  from injectivity of the right arrow, see
  Proposition~\ref{prop:forgetfulmap}, and injectivity of
  the bottom arrow, which has already been proven.
\end{proof}
 
%%%%%%%%%%%%%%%%%%%%%%%%%%%%%%%%%%%%%%%%%%%%%%%%%%%%%%%%%%%%%%%%%%%%%%%%%%%%%

\section{Fixed points} 
 
Let~$T$ be a trivial~$G$-space, and~let~$M$ be an
oriented~$W$-dimensional~$G$-manifold as before. There is a
map from~$\{T,S^{W^G}\}$ to~$\{T,S^{W^G}\}^G_{\mathrm{sf}}$,
using the fact that any map between trivial~$G$-spaces is
a~$G$-map, and there is a map
from~$\{T,S^{W^G}\}^G_{\mathrm{sf}}$
to~$\{T,S^W\}^G_{\mathrm{sf}}$ induced by the inclusion
of~$S^{W^G}$ into~$S^W$ which is a~$G$-map. The composition
\begin{equation}\label{splitting_map} 
  \{T,S^{W^G}\}\longrightarrow\{T,S^W\}^G_{\mathrm{sf}} 
\end{equation} 
has a retraction, namely the fixed point map. The splitting
theorem implies that a complement for the image
of the group~$\{T,S^{W^G}\}$ in the group~$\{T,S^W\}^G_{\mathrm{sf}}$ is
isomorphic to~\hbox{$\{T,E
G_+\wedge_GS^W\}$}. (See~\cite{Lewis:Splitting}, Section 2,
for the version for incomplete universes needed here.) One
can use that to prove the following.
 
\begin{prop}\label{prop:G in}
  The group~$\{M_+^G,S^W\}^G_{\mathrm{sf}}$ is isomorphic to
  the~$\{M_+^G,S^{W^G}\}$, which is free abelian. The rank is
  the number of components of~$M^G$.
\end{prop}
 
\begin{proof}
  Since~$S^W$ is~$(n-1)$-connected, so is~$EG_+\wedge_GS^W$. The
  dimension of~$M_+^G$ is smaller than~$n$. It follows that the
  group~\hbox{$\{M_+^G,EG_+\wedge_GS^W\}$} is zero. Thus, the map~(\ref{splitting_map}) is an isomorphism in the case~$T=M_+^G$. This
  and the ordinary Hopf theorem imply
  that~$\{M_+^G,S^W\}^G_{\mathrm{sf}} \cong\{M_+^G,S^{W^G}\}$, which
  is free abelian of the indicated rank.
\end{proof}
 
%%%%%%%%%%%%%%%%%%%%%%%%%%%%%%%%%%%%%%%%%%%%%%%%%%%%%%%%%%%%%%%%%%%%%%%%%%%%%

\section{All points} 
 
Let~$M$ be an oriented~$W$-dimensional~$G$-manifold as
before. We are now in the position to use the information
gathered on the free points and on the fixed points in order
to determine the structure of the
source~$\{M_+,S^W\}^G_{\mathrm{sf}}$ of the ghost map.

\begin{prop}\label{prop:structure of the source}
  The group~$\{M_+,S^W\}^G_{\mathrm{sf}}$ is free abelian of
  rank one more than the number of components of~$M^G$.
\end{prop}

\begin{proof}
  The starting point for the calculation
  of~$\{M_+,S^W\}^G_{\mathrm{sf}}$ is the cofibre sequence
  \begin{displaymath} 
    M_+^G\longrightarrow M_+\longrightarrow M/M^G.
  \end{displaymath} 
  Mapping this into~$S^W$, we get a long exact sequence 
  \begin{displaymath}
    \cdots
    \longleftarrow\{M_+^G,S^W\}^G_{\mathrm{sf}}  
    \longleftarrow\{M_+,S^W\}^G_{\mathrm{sf}} 
    \longleftarrow\{M/M^G,S^W\}^G_{\mathrm{sf}} 
    \longleftarrow
    \cdots.
  \end{displaymath}
  The group in the middle is to be computed. On the left side, since
  the~$G$-space~$M/M^G$ is free and of smaller dimension than~$\Sigma
  S^W$, the next group on the
  left~\hbox{$\{\Sigma^{-1}(M/M^G),S^W\}^G_{\mathrm{sf}}$}, which is
  isomorphic to~$\{M/M^G,\Sigma S^W\}^G_{\mathrm{sf}}$, is trivial. On
  the right side, the diagram
  \begin{center}
    \mbox{ \xymatrix{
        \{M/M^G,S^W\}^G_{\mathrm{sf}}\ar[d]_{\text{forget}} &
        \{\Sigma M_+^G,S^W\}^G_{\mathrm{sf}}\ar[d]^{\text{forget}}\ar[l] \\
        \{M/M^G,S^W\} & \{\Sigma M_+^G,S^W\}\ar[l] } }
  \end{center}
  shows that the top map from $\{\Sigma
  M_+^G,S^W\}^G_{\mathrm{sf}}$ to
  $\{M/M^G,S^W\}^G_{\mathrm{sf}}$ is zero: the left map is
  injective by Proposition~\ref{prop:freepoints}, and the
  bottom right group is trivial by the dimension and
  connectivity of the spaces involved.

  To sum up, there is a short exact sequence
  \begin{displaymath}
    0\longleftarrow
    \{M_+^G,S^W\}^G_{\mathrm{sf}}\longleftarrow
    \{M_+,S^W\}^G_{\mathrm{sf}}\longleftarrow
    \{M/M^G,S^W\}^G_{\mathrm{sf}}\longleftarrow 
    0. 
  \end{displaymath}
  By Proposition~\ref{prop:G in}
  again,~$\{M_+^G,S^W\}^G_{\mathrm{sf}}\cong\{M_+^G,S^{W^G}\}$
  is free abelian of the indicated rank. Thus, the short
  exact sequence must be splittable.
\end{proof}

%%%%%%%%%%%%%%%%%%%%%%%%%%%%%%%%%%%%%%%%%%%%%%%%%%%%%%%%%%%%%%%%%%%%%%%%%%%%%

\section{The ghost map}\label{sec:7}
 
Let~$M$ be an oriented~$W$-dimensional~$G$-manifold as
before. Now that the structure of the source and the target
of the ghost map are known, it is time to study the map
itself.

\begin{prop}\label{prop:ghost map is injective}
  The ghost map~{\upshape (\ref{ghost map})} is injective
  with a cyclic cokernel of order~$|G|$.
\end{prop}
 
\begin{proof}
  One may compare the short exact sequence used in the proof
  of Proposition~\ref{prop:structure of the source} to the
  short exact sequence
  \begin{displaymath}
    0\leftarrow
    \{M_+^G,S^{W^G}\} \leftarrow 
    \{M_+,S^W\}\oplus\{M_+^G,S^{W^G}\} \leftarrow
    \{M/M^G,S^W\} \leftarrow 0,
  \end{displaymath}
  which is built by using the isomorphism
  between the group~$\{M/M^G,S^W\}$ and the group~$\{M_+,S^W\}$
  discussed before and the identity
  on~$\{M_+^G,S^{W^G}\}$. The two short exact sequences
  yield the rows in the diagram
  \begin{center}
    \mbox{ 
      \xymatrix{
        \{M_+^G,S^W\}^G_{\mathrm{sf}} \ar[d] & \{M_+,S^W\}^G_{\mathrm{sf}} \ar[d]\ar[l] &
        \{M/M^G,S^W\}^G_{\mathrm{sf}} \ar[d]\ar[l] \\ \{M_+^G,S^{W^G}\}
        & \{M_+,S^W\}\oplus\{M_+^G,S^{W^G}\} \ar[l] & \{M/M^G,S^W\},
        \ar[l] } }
  \end{center}
  which will now be used to compare both of them. The
  vertical arrow on the right is the forgetful map. This map
  was shown to be injective with a cyclic cokernel of
  order~$|G|$ in Proposition~\ref{prop:forgetfulmap}. The
  vertical map in the middle is the ghost map: it sends
  a~$G$-map~$f$ to the pair~$(f,f^G)$. The vertical arrow on
  the left is the isomorphism which sends~$f$ to~$f^G$ as
  discussed above. The diagram commutes, and the snake lemma
  implies the result.
\end{proof}
 
From what has already been shown, it is by now established
that the image of the group~$\{M_+,S^W\}^G_{\mathrm{sf}}$
under the ghost map is a subgroup of index~$|G|$ in the
direct sum \hbox{$\{M_+,S^W\}\oplus\{M_+^G,S^{W^G}\}$}.
This subgroup contains~$(|G|,0)$ and projects
onto~$\{M_+^G,S^{W^G}\}$.  But, this does not determine the
image. Additional information from the geometric situation
seems to be required. This will be supplied for in the
following final section.
 
%%%%%%%%%%%%%%%%%%%%%%%%%%%%%%%%%%%%%%%%%%%%%%%%%%%%%%%%%%%%%%%%%%%%%%%%%%%%%

\section{Examples of maps}\label{sec:examples}
 
Let~$M$ be an oriented~$W$-dimensional~$G$-manifold as
before. In this final section, the group~$\{M_+,S^W\}$,
which is free abelian on one generator by the ordinary Hopf
theorem, has a distinguished generator, namely the one that
preserves the orientations. The elements of this group can
hence be thought of as integers. Similarly, the fixed point
space of~$M$ decomposes into components:
\begin{displaymath}
    M^G=\coprod_{\alpha}M^G_\alpha.
\end{displaymath}
The dimension of any of the components~$M^G_\alpha$ agrees
with that of~$S^{W^G}$. Thus the
group~$\{(M^G_\alpha)_+,S^{W^G}\}$ has a distinguished
generator, too.  Collecting these together, the restriction
of an element in~$\{M_+,S^W\}^G_{\mathrm{sf}}$ to the fixed
point space gives a family of integers, one for
each~$\alpha$. Using these identifications, the ghost map
sends an equivariant map to a pair~$(x,y)$ consisting of an
integer~$x$ and a family~$y=(y_{\alpha})$ of integers.
 
With these preparations, we may now prove
Theorem~\ref{thm:congruences} from the introduction. The
proof works as in Example 3 from the introduction.

\begin{proof}
  Let us fix an~$\alpha$. For any point in~$M^G_\alpha$, a
  neighbourhood in~$M$ is~$G$-homeomorphic to
  the~$G$-representation~$W$. Collapsing the complement
  yields a~$G$-map~$f_\alpha$ from~$M_+$ to~$S^W$. Note that
  this map is the chosen generator of~$\{M_+,S^W\}$, and the
  restriction to~$(M^G_\alpha)_+$ is the chosen generator of
  the group~$\{(M^G_\alpha)_+,S^W\}$. On the other hand,
  for~\hbox{$\beta\neq\alpha$}, the collapse map sends the
  subspace~$(M^G_\beta)_+$ to a point. Thus, the
  corresponding element in~$\{(M^G_\beta)_+,S^W\}$ is
  zero. Thus if~$1_{\alpha}$ denotes the characteristic
  function of~$\alpha$, the element
  of~$\{M_+,S^W\}^G_{\mathrm{sf}}$ which is represented
  by~$f_{\alpha}$ is mapped to~$(1,1_{\alpha})$.
  
  Now for each~$\alpha$ there has been produced
  a~$G$-map~$f_{\alpha}$ in~$\{M_+, S^W\}^G_{\mathrm{sf}}$
  which the ghost map sends to a pair~$(x,y)=(1,1_\alpha)$
  satisfying the relation in the theorem.  The subgroup of
  all pairs satisfying that relation is the unique subgroup
  of index~$|G|$ which contains the pairs~$(1,1_{\alpha})$
  for all~$\alpha$.
\end{proof}
  
%%%%%%%%%%%%%%%%%%%%%%%%%%%%%%%%%%%%%%%%%%%%%%%%%%%%%%%%%%%%%%%%%%%%%%%%%%%%%

\vfill

\parbox{\linewidth}{%
Markus Szymik\\
Department of Mathematical Sciences\\
NTNU Norwegian University of Science and Technology\\
7491 Trondheim\\
NORWAY\\
\href{mailto:markus.szymik@ntnu.no}{markus.szymik@ntnu.no}\\
\href{https://folk.ntnu.no/markussz}{folk.ntnu.no/markussz}}

\end{document}